\documentclass[10pt,a4paper]{amsart}
\usepackage[colorlinks=true, pdfstartview=FitV, linkcolor=blue, 
citecolor=blue]{hyperref}

\usepackage{amssymb,amscd}
\usepackage{microtype,xcolor,enumerate,comment}
\usepackage{a4wide}

\theoremstyle{plain}
\newtheorem{theorem}{Theorem}
\newtheorem{lemma}{Lemma}
\newtheorem{question}{Question}

\begin{document}

\title[Linear independence of series related to the Thue--Morse sequence]{Linear independence of series related to the Thue--Morse sequence along powers}

\author{Michael Coons}
\author{Yohei Tachiya}
\address{Department of Mathematics and Statistics,\newline
\hspace*{\parindent}California State
  University, Chico, CA 95929, USA} \email{mjcoons@csuchico.edu}
\address{Graduate School of Science and Technology,\newline
\hspace*{\parindent}Hirosaki University, Hirosaki 036-8561, Japan}
\email{tachiya@hirosaki-u.ac.jp}

\makeatletter
\@namedef{subjclassname@2020}{%
  \textup{2020} Mathematics Subject Classification}
\makeatother



\date{\today}

\begin{abstract} 
The Thue--Morse sequence $\{t(n)\}_{n\geqslant 1}$ is the indicator function of the parity of the number of ones in the binary expansion of positive integers $n$, 
where $t(n)=1$ (resp.~$=0$) if the binary expansion of $n$ has an odd (resp.~even) number of ones. 
In this paper, we generalize a recent result of E.~Miyanohara by showing that, for a fixed Pisot or Salem number $\beta>\sqrt{\varphi}=1.272019649\ldots$, 
the set of the numbers 
\[
1,\quad 
\sum_{n\geqslant  1}\frac{t(n)}{\beta^{n}},\quad 
\sum_{n\geqslant  1}\frac{t(n^2)}{\beta^{n}},\quad 
\dots, \quad
\sum_{n\geqslant  1}\frac{t(n^k)}{\beta^{n}},\quad \dots
\]
is linearly independent over the field $\mathbb{Q}(\beta)$, 
where $\varphi:=(1+\sqrt{5})/2$ is the golden ratio. 
Our result implies that for any $k\geqslant  1$ and for any $a_1,a_2,\ldots,a_k\in\mathbb{Q}(\beta)$, not all zero, 
the sequence \{$a_1t(n)+a_2t(n^2)+\cdots+a_kt(n^k)\}_{n\geqslant 1}$ cannot be eventually periodic. 
\end{abstract}

\maketitle

\section{Introduction}

Let $s_2(n)$ denote the number of ones in the binary expansion of $n$. 
The {\em Thue--Morse sequence} $t$ is defined, for $n\geqslant  1$, by $t(n)=1$ if $s_2(n)$ is odd, and $t(n)=0$ if $s_2(n)$ is even. 
The Thue--Morse sequence is paradigmatic in the areas of complexity and symbolic dynamics, and as such is an object of current interest in a variety of areas. 
While the sequence goes back at least to the 1851 paper of Prouhet \cite{P1851}, its interest in the context of complexity is usually attributed to Thue \cite{T1906}, 
who, in 1906, showed that $t$ is cube-free; that is, viewing $t$ as a one-sided infinite word 
$t(1)t(2)t(3)\cdots$, it contains no three consecutive identical binary blocks $aaa$. 
In particular, this shows that the Thue--Morse sequence is not periodic, though the proof of non-periodicity is much less deep than the cube-freeness. 
The non-periodicity implies that the number $\sum_{n\geqslant  1}t(n)b^{-n}$ is irrational for all integers $b\geqslant  2$, but more strongly, 
a result of Mahler \cite{M1930} from 1930 provides the transcendence 
of $\sum_{n\geqslant  1}t(n)\alpha^{-n}$ for any algebraic number $\alpha$ with $|\alpha|>1$. 
In this direction, recently, Miyanohara \cite{M2023} 
showed that if $\beta>2$ is a Pisot or Salem number, then the number $\sum_{n\geqslant  1}t(n^2){\beta^{-n}}$ does not belong to the field $\mathbb{Q}(\beta)$. 

In this paper, we generalize Miyanohara's results by proving the following theorem. 
Throughout the paper, $\varphi:=(1+\sqrt{5})/2$ denotes the golden ratio. 
Recall that an infinite set of numbers is called linearly independent if every non-empty finite subset is linearly independent. 

\begin{theorem}\label{thm:main} 
Let $\beta$ be a Pisot or Salem number with $\beta>\sqrt{\varphi}=1.272019649\ldots$. 
Then the numbers 
\[
1,\quad \sum_{n\geqslant  1}\frac{t(n)}{\beta^n},\quad 
\sum_{n\geqslant  1}\frac{t(n^2)}{\beta^n},\quad \ldots,\quad \sum_{n\geqslant  1}\frac{t(n^k)}{\beta^n},\quad \ldots
\] 
are linearly independent over $\mathbb{Q}(\beta).$ 
In particular, the result holds for all Pisot numbers.
\end{theorem}
%%%%%%%%%%%%%%%%%%%%%%%
Let $\beta$ be as in Theorem~\ref{thm:main}. 
Theorem \ref{thm:main} asserts that the sequence of the non-trivial $\mathbb{Q}(\beta)$-linear combination of the Thue--Morse sequence along powers
\begin{equation}\label{eq:1211}
s(n):=a_1t(n)+a_2t(n^2)+\cdots+a_kt(n^k)
\end{equation}
can not be eventually periodic, and that the number $\gamma:=\sum_{n\geqslant  1}s(n)\beta^{-n}$ 
does not belong to $\mathbb{Q}(\beta)$. %, for $\beta$ as in Theorem~\ref{thm:main}.  
Moreover, combining the result of Schmidt~\cite{S1980}, we find that   
the $\beta$-expansion of $\gamma$ is not eventually periodic. 
In particular, $\gamma_k:=\sum_{n\geqslant1}^{}t(n^k)\beta^{-n}\notin\mathbb{Q}(\beta)$ for every integer $k\geqslant1$  and 
the $\beta$-expansion of $\gamma_k$ is not eventually periodic. 

The Thue--Morse sequence along powers has been studied by several authors, with the most focus on 
the sequence $\{t(n^2)\}_{n\geqslant 1}$. 
In 2007, Moshe~\cite{YM2007} showed that 
the sequence $\{t(n^2)\}_{n\geq1}$ contains all finite binary words, namely, 
every finite word $w_1w_2\cdots w_m$ $(w_i\in\{0,1\})$ of length $m$ appears in the infinite word $t(1)t(4)t(9)\cdots$ 
(see Section~\ref{sec4} for more details). 
Moreover, this result was generalized by Drmota, Mauduit and Rivat \cite{DMR2019} who 
proved that the sequence $\{t(n^2)\}_{n\geqslant 1}$ is normal. 
As a consequence, the number $\sum_{n\geqslant 1}t(n^2)2^{-n}$ is normal in base $2$. 

The present paper is organized as follows. 
In Section \ref{sec2}, for each 
integers $r=1,2,\dots,k$, we will investigate the appearance of zeros %
in the sequences defined by the difference of $t(n^r)$ and a certain shift $t((n+n_0)^r)$.  
This observation makes it possible to find a good rational approximation to 
\[
\xi:=\sum_{n\geqslant  1}s(n)\beta^{-n},
\]
where the sequence $\{s(n)\}_{n\geq1}$ is defined in \eqref{eq:1211}. 
Section \ref{sec3} is devoted to the proof of Theorem~\ref{thm:main}. 
Finally, in Section \ref{sec4}, we will give remarks and further questions related to our results.

%%%%%%%%%%%%%%%%%%%%%%%%%%%%%%%%%%%%%%%%%%%%%%%%%%%%%%%%%%%%%%%%%%%%%%%%%%%%%
%%%%%%%%%%%%%%%%%%%%%%%%%%%%%%%%%%%%%%%%%%%%%%%%%%%%%%%%%%%%%%%%%%%%%%%%%%%%%%

\section{Some useful equalities and unequalities}\label{sec2}

The first lemma allows us to optimize our choice of rational approximations.

\begin{lemma}\label{lem:1104}
For any integer $k\geqslant 2$, there exist positive integers $m$ and $n$ such that the system of simultaneous congruences
\begin{equation}\label{cong}
\begin{array}{rl}
X\equiv&\!\!\! 2^{2m-1}-1\pmod{2^{2m}},\\
3^{k-1}X\equiv&\!\!\! 2^{2n}-1\pmod{2^{2n+1}}
\end{array}
\end{equation}
has an integer solution. 
\end{lemma}

\begin{proof}
For an even integer $k\geqslant 2$, clearly $x=1$ is a solution of \eqref{cong} with $m=n=1$. Thus, let $k=2\ell+1\geqslant  3$ be odd. 
Then there exist an integer $s\geqslant 3$ and $a\in\{0,1\}$ satisfying
\begin{equation*}
3^{k-1}=9^{\ell}\equiv 1+2^s+a2^{s+1}\pmod {2^{s+2}}.
\end{equation*}
If $s=2u+1\geqslant  3$ is odd,  we set $x:=2^{2u+1}-1+(1-a)2^{2u+2}$, so that 
\begin{align*}
3^{k-1} x&\equiv (1+2^{2u+1}+a2^{2u+2})(2^{2u+1}-1+(1-a) 2^{2u+2})\pmod{2^{2u+3}}\\
&\equiv 2^{2u+2}-1\pmod{2^{2u+3}},
\end{align*}
and hence, the integer $x$ is a solution of \eqref{cong} with $m=n=u+1$.
If $s=2u\geqslant  4$ is even, then setting $x:=2^{2u+1}-1+2^{2u+2}$, we obtain 
\begin{align*}
3^{k-1} x&\equiv (1+2^{2u})(-1)\equiv 2^{2u}-1\pmod {2^{2u+1}}, 
\end{align*}
and the integer $x$ is a solution of \eqref{cong} with $m=u+1$ and $n=u$. 
Lemma~\ref{lem:1104} is proved. 
\end{proof}
%%%%%%%%%%%%%%%%%%%%%%%%%%%%%%%%%%%%%%%%%%%%%%%%%%%%%%%%%%%%%%%%%%%%%%%%%%%%%%%%%%%%%%%

Let $x\geqslant 1$ be an integer solution of the system of congruences \eqref{cong}. 
Then the binary expansions of the integers $x$ and $x3^{k-1}$ have the forms 
\[
(x)_2=w_10\overbrace{11\cdots1}^{2m-1}\qquad {\rm and}\qquad 
(x3^{k-1})_2=w_20\overbrace{11\cdots1}^{2n},
\]  
respectively, for some binary words $w_1$ and $w_2$, 
and hence, we obtain the equalities 
\begin{equation}\label{eq:rem}
t(1+x)=t(x)\qquad{\rm and} \qquad t(1+x3^{k-1})=1-t(x3^{k-1}).
\end{equation}

Let $k\geqslant 2$ be an integer and $\nu(k):=\nu_2(k)$ be the $2$-adic valuation of $k$. 
Fix positive integers 
$m$, $n$, $x$ (depending only on $k$) as in Lemma~\ref{lem:1104}. 
Since $k2^{-\nu(k)}$ is an odd integer, the congruence 
\begin{equation}\label{eq:1105}
k2^{-\nu(k)}\,Y\equiv x\pmod{2^{2m+2n+1}}
\end{equation}
has an integer solution. Let $y\leqslant  2^{2m+2n+1}$ be the least positive integer solution of~\eqref{eq:1105}. 

The remaining lemmas give us quantitative information about our approximations, 
as well as allowing us to optimize the range of $\beta$ in Theorem \ref{thm:main}.
Throughout the paper, let $N$ be a sufficiently large integer.  
In the following Lemmas~\ref{lem:144} and \ref{lem:113}, we investigate the Thue--Morse values of the integers  
\begin{equation}\label{eq:903}
(y2^{kN-\nu(k)+\delta}+j)^k=
\sum_{\ell=0}^{k}\binom{k}{\ell}y^\ell j^{k-\ell}\cdot 2^{(kN-\nu(k)+\delta)\ell},\qquad \delta\in\{0,1\}
\end{equation}
for $j=0,1,\dots,2^{N}+4$ and $j=2^N+2h$ $(h=3,4,\dots,2^{N-3})$. 
Define  
\[
A_{\ell,j}:=\binom{k}{\ell}y^\ell j^{k-\ell},\qquad \ell=0,1,\dots,k,\quad j=1,2,\dots,2^N+2^{N-2}.
\]
When $1\leqslant  \ell\leqslant  k$, we have 
\begin{equation}\label{eq:3125}
A_{\ell,j}\leqslant  2^k\cdot(2^{2m+2n+1})^k\cdot (5\cdot 2^{N-2})^{k-1}<2^{kN-\nu(k)+\delta}
\end{equation}
since $N$ is sufficiently large. 
Hence, using \eqref{eq:903} and \eqref{eq:3125}, we obtain 
\begin{align}
t\left(
(y2^{kN-\nu(k)+\delta}+j)^k
\right)&\equiv 
t\left(A_{0,j}+A_{1,j}2^{kN-\nu(k)+\delta}\right)+\sum_{\ell=2}^{k}t\left(A_{\ell,j}2^{(kN-\nu(k)+\delta)\ell}\right)\nonumber\\
&=t\left(j^k+zj^{k-1}2^{kN+\delta}\right)+\sum_{\ell=2}^{k}t\left(A_{\ell,j}\right)\pmod{2},
\label{eq:8469}
\end{align}
where $z:=k2^{-\nu(k)}y$. Note that the integers $A_{\ell,j}$ are independent of $\delta$.

%%%%%%%%%%%%%%%%%%%%%%%%%%%%%%%%%%%%%%%%%%%%%%%%%%%%%%%%%%%%%%%%%%%%%%%%%%%%%%%%%%%%%%%%%%%%%%%%
\begin{lemma}\label{lem:144}
For every integer  $j=0,1,\dots,2^N-1$ and $j=2^N+2h$ $(h=0,1,\dots,$ $2^{N-3})$, 
we have 
\begin{equation}\label{eq:221}
t\left((y2^{kN-\nu(k)}+j)^k\right)=
t\left((y2^{kN-\nu(k)+1}+j)^k\right). 
\end{equation}
\end{lemma}
%%%%%%%%%%%%%%%%%%%%%%%%%%%%%%%%%%%%%%%%%%%%%%%%%%%%%%%%%%%%%%%%%%%%%%%%%%%%%%%%%%%%%%%%%%%%
\begin{proof} 
Let $\delta\in\{0,1\}$. 
For $j=0,1,\dots,2^N-1$, we have 
\begin{equation}\label{eq:313}
j^k<2^{kN}\qquad {\rm and}\qquad 
2^{kN}\mid zj^{k-1}2^{kN+\delta},
\end{equation}
and moreover, for $j=2^N+2h$ $(h=0,1,\dots,2^{N-3})$,
\begin{equation}\label{eq:3198}
j^k\leqslant (5\cdot 2^{N-2})^k<2^{kN+k-1}\quad\mbox{and}\quad
2^{kN+k-1}\mid zj^{k-1}2^{kN+\delta},
\end{equation} since $zj^{k-1}2^{kN+\delta}=z(2^{N-1}+h)^{k-1}2^{kN+k-1+\delta}.$
Hence, by \eqref{eq:313} and \eqref{eq:3198}, 
\begin{equation}\label{eq:648}
t(j^k+zj^{k-1}2^{kN+\delta})\equiv t(j^k)+t(zj^{k-1})\pmod{2},
\end{equation}
so that \eqref{eq:8469} and \eqref{eq:648} yield 
\begin{equation}\label{eq:2392}
t\left((y2^{kN-\nu(k)+\delta}+j)^k\right)\equiv  t(j^k)+t(zj^{k-1})
+\sum_{\ell=2}^{k}t(A_{\ell,j})\pmod{2}
\end{equation}
for $j=0,1,\dots,2^N-1$ and $j=2^N+2h$ $(h=0,1,\dots,2^{N-3})$. 
Thus, Lemma~\ref{lem:144} is proved since the right-hand side in \eqref{eq:2392} 
is independent of $\delta$. 
\end{proof}

Next we show that the equality~\eqref{eq:221} also holds for $j=2^N+3$, but not for $j=2^N+1$. 
%%%%%%%%%%%%%%%%%%%%%%%%%%%%%%%%%%%%%%%%%%%%%%%%%%%%%%%%%%%%%%%%%%%%%%%%%%%%%%%%%%%%%%
\begin{lemma}\label{lem:113}
We have
\begin{equation*}
t\left((y2^{kN-\nu(k)}+2^N+1)^k\right)\neq 
t\left((y2^{kN-\nu(k)+1}+2^N+1)^k\right),\end{equation*} and \begin{equation*}
t\left((y2^{kN-\nu(k)}+2^N+3)^k\right)=
t\left((y2^{kN-\nu(k)+1}+2^N+3)^k\right).
\end{equation*}
\end{lemma}
%%%%%%%%%%%%%%%%%%%%%%%%%%%%%%%%%%%%%%%%%%%%%%%%%%%%%%%%%%%%%%%%%%%%%%%%%%%%%%%%%%%%%%
\begin{proof} 
Let $\delta\in\{0,1\}$ and $j:=2^N+i$ $(i=1,3)$. 
Then we have
\begin{align}
j^k+zj^{k-1}2^{kN+\delta}&=
\sum_{\ell=0}^{k}\binom{k}{\ell}i^{k-\ell}\cdot 2^{N\ell}+\sum_{\ell=0}^{k-1}\binom{k-1}{\ell}zi^{k-1-\ell}\cdot 2^{N(\ell+k)+\delta}\nonumber\\
\label{eq:2749}&=
\sum_{\ell=0}^{k-1}\binom{k}{\ell}i^{k-\ell}\cdot 2^{N\ell}+
(1+zi^{k-1}2^{\delta})\cdot 2^{kN}\\
\nonumber&\qquad\qquad
+
\sum_{\ell=k+1}^{2k-1}\binom{k-1}{\ell-k}zi^{2k-\ell-1}\cdot 2^{N\ell+\delta}.
\end{align}
Since the integers 
\[
B_{\ell,i}:=\binom{k}{\ell}i^{k-\ell}, \quad 1+zi^{k-1}2^\delta, \quad C_{\ell,i}:=\binom{k-1}{\ell-k}zi^{2k-\ell-1}, 
\]
are independent of $N$, it follows from \eqref{eq:2749} that 
\begin{equation}\label{eq:0982647}
t\left(j^k+zj^{k-1}2^{kN+\delta}\right)\equiv\sum_{\ell=0}^{k-1}t(B_{\ell,i})+t\left(1+zi^{k-1}2^{\delta}\right)+
\sum_{\ell=k+1}^{2k-1}t(C_{\ell,i})
\pmod{2}.
\end{equation}
Moreover, combining Lemma~\ref{lem:1104} with \eqref{eq:1105}, we obtain 
\begin{align*}
zi^{k-1}=k2^{-\nu(k)}yi^{k-1}&\equiv xi^{k-1} \pmod{2^{2m+2n+1}}\\
&\equiv
\begin{cases}
2^{2m-1}-1\pmod{2^{2m}},&\mbox{if $i=1$},\\
2^{2n}-1\pmod{2^{2n+1}},&\mbox{if $i=3$},
\end{cases}
\end{align*}
so that by \eqref{eq:rem}
\begin{equation}\label{eq:9742}
t(1+zi^{k-1})=
\begin{cases}
t(zi^{k-1}),&\mbox{if $i=1$},\\
1-t(zi^{k-1}),&\mbox{if $i=3$}.
\end{cases}
\end{equation}
Thus, by \eqref{eq:8469}, \eqref{eq:0982647}, and \eqref{eq:9742}, we obtain 
%%%%
\begin{align*}
t\left(
(y2^{kN-\nu(k)+1}+j)^k
\right)-&
t\left(
(y2^{kN-\nu(k)}+j)^k
\right)\\
&\equiv 
t\left(j^k+zj^{k-1}2^{kN+1}\right)-t\left(j^k+zj^{k-1}2^{kN}\right)\\[0.2cm]
&\equiv t(1+2zi^{k-1})-t(1+zi^{k-1})\\[0.2cm]
&\equiv 
1+t(zi^{k-1})-
\begin{cases}
t(zi^{k-1})&\mbox{if $i=1$},\\
1-t(zi^{k-1})&\mbox{if $i=3$}
\end{cases}\\[0.2cm]
&\equiv
\begin{cases}
1&\mbox{if $j=2^N+1$},\\
0&\mbox{if $j=2^N+3$},
\end{cases}
\!\!\pmod{2},
\end{align*}
which finishes the proof of the Lemma~\ref{lem:113}. 
\end{proof}

%%%%%%%%%%%%%%%%%%%%%%%%%%%%%%%%%%%%%%%%%%%%%%%%%%%%%%%%%%%%%%%%%%%%%%%%%%%%%%%%%%%%%%%%%%%%%%%%%%%%%%%%%%%%%%%%%%%%%%%%%%%%%%%%
%%%%%%%%%%%%%%%%%%%%%%%%%%%%%%%%%%%%%%%%%%%%%%%%%%%%%%%%%%%%%%%%%%%%%%%%%%%%%%%%%%%%%%%%%%%%%%%%%%%%%%%%%%%%%%%%%%%%%%%%%%%%%%%%
Define 
\[
\lambda:=1+\frac{1}{2(k-1)}>1, 
\]
and let $\lfloor \alpha \rfloor$ denote the integral part of the real number $\alpha$. 
\begin{lemma}\label{lem:1431}
For every integer $r=1,2,\dots,k-1$, and for every integer $j=0,1,\dots,2^{\lfloor \lambda N\rfloor}$, 
we have 
\begin{equation*}\label{eq:2211}
t\left((y2^{kN-\nu(k)}+j)^r\right)=
t\left((y2^{kN-\nu(k)+1}+j)^r\right).
\end{equation*}
\end{lemma}
%%%%%%%%%%%%%%%%%%%%%%%%%%%%%%%%%%%%%%%%%%%%%%%%%%%%%%%%%%%%%%%%%%%%%%%%%%%%%%%%%%%%%%%%%%%%%%%%%%%%%%%%%%%%%%%%%%%%%%%%%%%%%%%%%%%%%%%%
\begin{proof} 
Let $\delta\in\{0,1\}$ and $r, j$ be fixed integers as in the lemma. Then we have   
\[
(y2^{kN-\nu(k)+\delta}+j)^r=\sum_{\ell=0}^{r}\binom{r}{\ell}y^{\ell}j^{r-\ell}\cdot 2^{(kN-\nu(k)+\delta)\ell}.
\]
Since 
\[
D_{\ell,j}:=\binom{r}{\ell}y^{\ell}j^{r-\ell}\leqslant  2^{k-1}\cdot(2^{2m+2n+1})^{k-1}\cdot (2^{\lambda N})^{k-1}<2^{kN-\nu(k)+\delta},
\]
we obtain 
\[
t\left((y2^{kN-\nu(k)+\delta}+j)^r\right)\equiv \sum_{\ell=0}^{r}t(D_{\ell,j})\pmod{2},
\]
which is independent of $\delta$. Lemma~\ref{eq:2211} is poved. 
\end{proof}

%%%%%%%%%%%%%%%%%%%%%%%%%%%%%%%%%%%%%%%%%%%%%%%%%%%%%%%%%%%%%%%%%%%%%%%%%%%%%%%%%%%%%%%%%%%%%%%%

\section{Proof of Theorem \ref{thm:main}}\label{sec3}

Let $\beta$ be a Pisot or Salem number with $\beta>\sqrt{\varphi}=1.272019649\ldots$. 
Suppose to the contrary that there exist an integer $k\geqslant 1$ and algebraic numbers  
$a_0,a_1,\ldots,a_k\in\mathbb{Q}(\beta)$, not all zero,  such that 
\[
a_0+\xi:=a_0+a_1\sum_{n\geqslant  1}\frac{t(n)}{\beta^n}+a_2\sum_{n\geqslant  1}\frac{t(n^2)}{\beta^n}+\cdots+a_k\sum_{n\geqslant  1}\frac{t(n^k)}{\beta^n}=0. 
\] 
We may assume that $a_0,a_1,\ldots,a_k\in\mathbb{Z}[\beta]$ and $a_k\neq0$. 
As mentioned in the introduction, the number $\sum_{n\geqslant  1}^{}t(n)\alpha^n$ is transcendental for any algebraic number $\alpha$ with $|\alpha|>1$, 
and thus, we have $k\geqslant 2$. 
Define the sequence $\{s(n)\}_{n\geqslant 1}$ by  
\begin{equation*}\label{s_n}
s(n):=a_1t(n)+a_2t(n^2)+\cdots+a_kt(n^k),\quad n\geqslant 1. 
\end{equation*}
Note that $\{s(n)\}_{n\geqslant 1}$ is bounded. 
Let $\nu(k),y$ be as in Section~\ref{sec2} and 
$N$ be a sufficiently large integer such that Lemmas \ref{lem:144},  \ref{lem:113}, and \ref{lem:1431} all hold.  
For convenience, let
\[
\kappa(N):=kN-\nu(k).
\]
Define the algebraic integers $p_N,q_N\in\mathbb{Z}[\beta]$, respectively, by 
\[
p_N:=(\beta^{y2^{\kappa(N)}}-1)\sum_{n=1}^{y2^{\kappa(N)}-1}s(n)\beta^{y2^{\kappa(N)}-n}+\sum_{n=y2^{\kappa(N)}}^{y2^{\kappa(N)+1}-1}s(n)\beta^{y2^{\kappa(N)+1}-n},
\]
and $q_N:=(\beta^{y2^{\kappa(N)}}-1)\beta^{y2^{\kappa(N)}}$. 
Then, we have 
\begin{align}
\frac{p_N}{q_N}&=\sum_{n=1}^{y2^{\kappa(N)-1}}\frac{s(n)}{\beta^{n}}+\frac{\beta^{y2^{\kappa(N)}}}{\beta^{y2^{\kappa(N)}}-1}\cdot
\sum_{n=y2^{\kappa(N)}}^{y2^{\kappa(N)+1}-1}\frac{s(n)}{\beta^{n}}\nonumber\\[0.2cm]
&=\sum_{n=1}^{y2^{\kappa(N)}-1}\frac{s(n)}{\beta^{n}}+
\left(
1+\frac{1}{\beta^{y2^{\kappa(N)}}}+\frac{1}{\beta^{y2^{\kappa(N)+1}}}+\cdots
\right)
\sum_{n=y2^{\kappa(N)}}^{y2^{\kappa(N)+1}-1}\frac{s(n)}{\beta^{n}}\nonumber\\[0.2cm]
&=
\sum_{n=1}^{y2^{\kappa(N)+1}-1}\frac{s(n)}{\beta^{n}}
+
\sum_{j=0}^{y2^{\kappa(N)}-1}\frac{s(y2^{\kappa(N)}+j)}{\beta^{y2^{\kappa(N)+1}+j}}
+O\left(
\frac{1}{\beta^{3y2^{\kappa(N)}}}
\right)\nonumber\\[0.2cm]
&
=
\sum_{n=1}^{y2^{\kappa(N)+1}+2^N}\frac{s(n)}{\beta^{n}}
+
\sum_{j=2^N+1}^{2^{\lfloor \lambda N\rfloor}}
\frac{s(y2^{\kappa(N)}+j)}{\beta^{y2^{\kappa(N)+1}+j}}
+O\left(
\frac{1}{\beta^{y2^{\kappa(N)+1}+2^{\lfloor \lambda N\rfloor}}}
\right),\label{eq:1209}
\end{align}
where, for the last equality, we used Lemmas \ref{lem:144} and \ref{lem:1431}, and the equalities 
\[
s(y2^{\kappa(N)}+j)=s(y2^{\kappa(N)+1}+j),\qquad j=0,1,\ldots,2^N-1. 
\] 
Hence, by \eqref{eq:1209} and Lemma~\ref{lem:1431} again, 
\begin{equation}\label{eq:Aktmk}
\xi-\frac{p_N}{q_N}=
a_k\sum_{j=2^N+1}^{2^{\left\lfloor \lambda N\right\rfloor}}
\frac{u(j)
}{\beta^{y2^{\kappa(N)+1}+j}} +
O\left(\frac{1}{\beta^{y2^{\kappa(N)+1}+ 2^{\left\lfloor\lambda N\right\rfloor}}}\right),
\end{equation} 
where $u(j)$ is defined by 
\[
u(j):=t\left((y2^{\kappa(N)+1}+j)^k\right)-t\left((y2^{\kappa(N)}+j)^k\right).
\]
Note that $|u(j)|\leqslant  1$ since $u(j)\in\{-1,0,1\}$ for every integer $j$. 
Moreover, by the definition of $q_N$, we have $q_N\leqslant  \beta^{y2^{\kappa(N)+1}}$, and so 
by \eqref{eq:Aktmk} 
\begin{equation}\label{eq:upperbound}
q_N\xi-p_N=O\left(\frac{1}{\beta^{2^N}}\right).
\end{equation}

On the other hand, applying Lemmas \ref{lem:144} and \ref{lem:113}, we obtain  
\begin{align}
\left|
\sum_{j=2^N+1}^{2^{\left\lfloor \lambda N\right\rfloor}}
\frac{u(j)}{\beta^{j}}\right|
&
\geqslant 
\frac{1}{\beta^{2^N+1}}-
\sum_{j=2^N+5}^{2^{\left\lfloor \lambda N\right\rfloor}}
\frac{|u(j)|
}{\beta^{j}}\nonumber \\[0.2cm]
&=
\frac{1}{\beta^{2^N+1}}
-
\sum_{
\substack{j\geqslant  2^N+5\\j{\rm :odd}}
}^{}
\frac{1
}{\beta^{j}}
-
\sum_{
\substack{
j\geqslant  2^N+2^{N-2}+1\\j{\rm :even}}
}^{}
\frac{1
}{\beta^{j}}
\nonumber\\[0.2cm]
%%%%%%%%%%%%%%%%%%%%%%%%%%%%%%%%%%%%%%%%%%%%%%%%%%%%%%%%%%%%%%%%%%%%%%%%%%%%%%%%%%%%%%%%%%%%%%%%%
&\geqslant 
\frac{1}{\beta^{2^N+1}}-
\frac{\beta^2}{\beta^2-1}
\left( 
\frac{1}{\beta^{2^N+5}}
+
\frac{1}{\beta^{5\cdot 2^{N-2}+2}}
\right)\nonumber\\[0.2cm]
&=
\frac{\beta^4-\beta^2-1}{\beta^2(\beta^2-1)}\cdot\frac{1}{\beta^{2^N}}+O\left(
\frac{1}{\beta^{5\cdot 2^{N-2}}}
\right),\label{eq:111902}
\end{align}
and hence, by \eqref{eq:Aktmk} and \eqref{eq:111902}, 
\begin{align*}
\beta^{y2^{\kappa(N)+1}}
\left|
\xi-\frac{p_N}{q_N}\right|&
\geqslant   
|a_k|\cdot \left|
\sum_{j=2^N+1}^{2^{\left\lfloor \lambda N\right\rfloor}}
\frac{u(j)
}{\beta^{j}}
\right|+
O\left(
\frac{1}{
\beta^{2^{\left\lfloor\lambda N\right\rfloor}}
}
\right)\\[0.2cm]
&=
|a_k|\cdot\frac{\beta^4-\beta^2-1}{\beta^2(\beta^2-1)}\cdot\frac{1}{\beta^{2^N}}+O\left(
\frac{1}{\beta^{5\cdot 2^{N-2}}}
\right)+O\left(
\frac{1}{
\beta^{2^{\left\lfloor\lambda N\right\rfloor}}
}
\right).
\end{align*}
Thus, noting that $\beta>\sqrt{\varphi}$, $a_k\neq0$, and that both $5\cdot 2^{N-2}-2^N$ and $2^{\lfloor \lambda N\rfloor}-2^N$ tend to infinity as $n\to \infty$, we obtain 
\begin{equation}\label{eq:lowerbound} 
\left|\xi-\frac{p_N}{q_N}\right|>0.
\end{equation}
Combining \eqref{eq:upperbound} and \eqref{eq:lowerbound}, 
there is a positive constant $c_1$ such that 
\begin{equation}\label{est}
0<\left|q_N\xi-p_N\right|<\frac{c_1}{\beta^{2^N}}
\end{equation}
for every sufficiently large $N$.  

Now, we complete the proof.   
When $\beta$ is a rational integer, \eqref{est} is clearly impossible for large $N$, 
since $q_N\xi-p_N$ is a non-zero rational integer. 
So, suppose not, and let $\beta=:\beta_1,\beta_2,\ldots,\beta_d$ $(d\geqslant 2)$ be the Galois conjugates over $\mathbb{Q}$ of $\beta$. 
Since $\xi=-a_0\in\mathbb{Z}[\beta]$, there exist rational integers $A_0,A_1,\ldots,A_{d-1}$ 
such that 
$
\xi=\sum_{i=0}^{d-1}{A_i}\beta^i.
$ 
Define the polynomial over $\mathbb{Z}[\beta]$
\[
F_N(X):=q_N(X)\xi(X)-p_N(X)
\]
where 
\begin{align*}
p_N(X)&:=(X^{y2^{\kappa(N)}}-1)\sum_{n=1}^{y2^{\kappa(N)}-1}s(n)X^{y2^{\kappa(N)}-n}+\sum_{n=y2^{\kappa(N)}}^{y2^{\kappa(N)+1}-1}s(n)X^{y2^{\kappa(N)+1}-n},\\
q_N(X)&:=(X^{y2^{\kappa(N)}}-1)X^{y2^{\kappa(N)}},\\
\xi(X)&:=\sum_{i=0}^{d-1}A_iX^i.
\end{align*}
Note that $p_N(\beta)=p_N$, $q_N(\beta)=q_N$, and $\xi(\beta)=\xi$. By \eqref{est},  
\begin{equation}\label{eq:pnbetasmall}
0<|F_N(\beta)|=\left|q_N\xi-p_N\right|<\frac{c_1}{\beta^{2^N}}.
\end{equation} 
Moreover, since $\beta$ is a Pisot or Salem number, we have 
$|\beta_i|\leqslant  1$ $(i=2,3,\dots,d)$, and so by the definitions of $p_N(X),q_N(X),\xi(X)$ 
there exists a positive constant $c_2$ independent of $N$ such that
\begin{equation}\label{fin}
0<|F_N(\beta_i)|\leqslant  c_2\cdot y2^{\kappa(N)},\qquad i=2,3,\dots,d,
\end{equation}
where the first inequality follows 
since $F_N(\beta_i)$ are the Galois 
conjugates of $F_N(\beta)\neq0$. 
Therefore, considering the norm over $\mathbb{Q}(\beta)/\mathbb{Q}$ of the algebraic integer $F_N(\beta)$, 
we obtain by \eqref{eq:pnbetasmall} and \eqref{fin},
\[
1\leqslant  |N_{\mathbb{Q}(\beta)/\mathbb{Q}}F_N(\beta)|
=\prod_{i=1}^{d-1}|F_N(\beta_i)|
< \frac{c_1 c_2^{d-1}\cdot (y2^{\kappa(N)})^d}{\beta^{2^N}},
\]
which is impossible for sufficiently large $N$,  
since $\kappa(N)=O(N)=o(2^N)$. 
The proof of Theorem~\ref{thm:main} is now complete.

%%%%%%%%%%%%%%%%%%%%%%%%%%%%%%%%%%%%%%%%%%%%%%%%%%%%%%%%%%%%
\section{Concluding remarks and further questions}\label{sec4}

The careful reader will notice that, while the range of our results includes all Pisot numbers, it does not include all Salem numbers; for example, we miss the smallest known Salem number, Lehmer's number $\lambda\approx1.17628018$. As there is no known nontrivial lower bound on Salem numbers, for our theorem to apply to all Salem numbers, we would need our result to be valid for $\beta>1$. It seems very unlikely that the type of optimization that we have done here could be carried out to reach that range. A similar approach could increase the range a bit, but a new idea is probably necessary to get the full range of possible $\beta$. Here, our proof works for Pisot and Salem numbers, but it seems reasonable to conjecture that Theorem \ref{thm:main} holds for any algebraic number $\beta$ with $|\beta|>1$, though without more information on the structure of the sequences $\{t(n^k)\}_{n\geqslant 1}$ this seems out of reach at the moment.

When we first started our investigation, we wanted to show that the three numbers 
\[
1,\ \ \sum_{n\geqslant  1}\frac{t(n)}{b^n},\ \mbox{and} \ \sum_{n\geqslant  1}\frac{t(n^2)}{b^n}
\]
are linearly independent over $\mathbb{Q}$ for any positive integer $b\geqslant  2$. Considering this question, two properties of $t$ stood out to us. First, the sequence $\{t(n)\}_{n\geqslant 1}$ is produced by a finite automaton, so it is not a very complicated sequence. Second, the sequence $\{t(n^2)\}_{n\geqslant 1}$ is extremely complicated---
Drmota, Mauduit and Rivat \cite{DMR2019} showed that the sequence $\{t(n^2)\}_{n\geqslant 1}$ is normal to the base $2$, that is, all $2^m$ patterns of length $m$ occur for each $m$, and they occur with frequency $2^{-m}$---a classical result of Wall \cite{W1950} then implies the $\mathbb{Q}$-linear independence of the above three numbers when $b=2$.  It seems reasonable to think that for any nonzero rational numbers $q_1$ and $q_2$, the number 
\[
q_2+q_1\sum_{n\geqslant  1}\frac{t(n)}{b^n}
\]
must have a `not very complicated' base-$b$ expansion. At the level of sequences, results like this are known, but it seems at the level of expansions of real numbers there is a gap in the literature here. We make explicit a question that would be a first step in this direction.

Recall, for any sequence $f$ taking values in $\{0,1\}$, we let $p_f(m)$ denote the {\em (subword) complexity} of $f$ as a one-sided infinite word. 
In particular, $p_f(m)$ counts the number of distinct blocks of length $m$ in $f$. 
So, for example, $f$ is eventually periodic if and only if $p_f(m)$ is uniformly bounded, 
and if $f$ is $2$-normal then $p_f(m)=2^m$ for all $m$, since every binary word of any length appears in a $2$-normal binary word. 
The {\em entropy} of $f$ is the limit, $h(f):=\lim_{m\to\infty}(\log p_f(m))/m\in[0,\log 2]$. 
Considering the Thue--Morse sequence (or word) $t$, since $t$ is generated by a finite automaton, 
we have that $p_t(m)=O(m)$, so $h(t)=0$. In the present paper, we considered the sequences $t^k:=\{t(n^k)\}_{n\geqslant 1}.$ 
Moshe \cite{YM2007} established that $p_{t^k}(m)\geqslant  2^{m/2^{k-2}}$ for any $k\geqslant  2$; 
that is, that $h(t^k)\geqslant  \log(2)/2^{k-2}>0$.  
The question about numbers with lower complexity seems immediate. 

\begin{question} For a real number $\xi$, let $p_\xi(m,b)$ be the number of $b$-ary words of length $m$ appearing in the base-$b$ expansion of $\xi$. 
Is it true that for any two rational numbers $q_1$ and $q_2$, we have $p_{q_1\xi+q_2}(m,b)=O(p_\xi(m,b))\ ?$
\end{question}

%%%%%%%%%%%%%%%%%%%%%%%%%%%%%%%%%%%%%%%%%%%%%%
\section*{Acknowledgements}
%%%%%%%%%%%%%%%%%%%%%%%%%%%%%%%%%%%%%%%%%%%%%%

This work began during a visit to Hirosaki University. Michael Coons thanks Hirosaki Unversity for their support, and their staff and students for their outstanding hospitality.  This work was supported by the Research Institute for Mathematical Sciences, an International Joint Usage/Research Center located in Kyoto University. This research was partly supported by JSPS KAKENHI Grant Number JP22K03263. 

%%%%%%%%%%%%%%%%%%%%%%%%%%%%%%%%%%%%%%%%%%%%%%%%
\bibliographystyle{amsplain}
\providecommand{\bysame}{\leavevmode\hbox to3em{\hrulefill}\thinspace}
\providecommand{\MR}{\relax\ifhmode\unskip\space\fi MR }
\providecommand{\MRhref}[2]{%
  \href{http://www.ams.org/mathscinet-getitem?mr=#1}{#2}
}
\providecommand{\href}[2]{#2}

\end{document}